\documentclass[a4paper,14pt,reqno,intlimits]{extarticle}

\usepackage[cp1251]{inputenc}
\usepackage[T2A]{fontenc}
\usepackage[english,russian]{babel}
\usepackage{amsfonts,amsmath,amsxtra,amsthm,amssymb,latexsym}
\usepackage{mathrsfs}
\usepackage{graphicx}
\usepackage{setspace}
\usepackage{multicol}
\usepackage[ruled]{algorithm2e}

\SetAlCapHSkip{0pt}
\IncMargin{-\parindent}

\theoremstyle{plain} 
\newtheorem{theorem}{Теорема}

\theoremstyle{definition} 
\newtheorem{remark}{Замечание}
\newtheorem{example}{Пример}

\numberwithin{equation}{section} \numberwithin{theorem}{section}
\numberwithin{lemma}{section} \numberwithin{proposition}{section}
\numberwithin{corollary}{section} \numberwithin{remark}{section}
\numberwithin{definition}{section} \numberwithin{example}{section}

\newcommand*{\affaddr}[1]{#1} 
\newcommand*{\affmark}[1][*]{\textsuperscript{#1}}

\begin{document}

\renewcommand{\thesection}{}
\renewcommand{\thesubsection}{\arabic{subsection}}
\renewcommand{\thetheorem}{\arabic{theorem}}
\renewcommand{\theequation}{\arabic{equation}}
\renewcommand{\thedefinition}{\arabic{definition}}
\renewcommand{\thecorollary}{\arabic{corollary}}
\renewcommand{\thelemma}{\arabic{lemma}}
\renewcommand{\theexample}{\arabic{example}}
\renewcommand{\theremark}{\arabic{remark}}

\title{\textbf{Адаптивный проксимальный метод для вариационных неравенств}}
\author{
\emph{Ф.~С.~Стонякин\affmark[1, 2], А.~В.~Гасников\affmark[2, 3, 4]},\\ \emph{П.~Е.~Двуреченский\affmark[4, 5], А.~А.~Титов\affmark[2]}\\\\\small
\affaddr{\affmark[1]\textbf{Крымский федеральный университет имени В.~И.~Вернадского},\\\small пр-кт Академика Вернадского, 4, Симферополь, РФ, 295007;}\\\small
\affaddr{\affmark[2]\textbf{Московский физико-технический институт},\\\small Институтский пер., 9, Долгопрудный, РФ, 141070;} \\\small
\affaddr{\affmark[3]\textbf{Национальный исследовательский университет "Высшая школа экономики"}, \\\small Кочновский проезд, 3, Москва, РФ, 125319;}\\\small
\affaddr{\affmark[4]\textbf{Институт проблем передачи информации имени А.~А.~Харкевича РАН},\\\small Большой Каретный пер., 19, Москва, РФ, 127051;}\\\small
\affaddr{\affmark[5]\textbf{Weierstrass Institute for Applied Analysis and Stochastics},\\\small Mohrenstr. 39, Berlin, Germany, 10117.}\\
}
\date{}
\maketitle

\begin{abstract}
В статье предложен новый аналог проксимального зеркального метода А.~С.~Немировского с адаптивным выбором констант в минимизируемых прокс-отображениях на каждой итерации для вариационных неравенств с липшицевым полем. Получены оценки необходимого числа итераций для достижения заданного качества решения вариационного неравенства. Показано, как можно обобщить предлагаемый подход на случай гёльдерова поля. Рассмотрена модификация предлагаемого алгоритма в случае неточного оракула для оператора поля.
\end{abstract}

Вариационные неравенства (ВН) нередко возникают в самых разных проблемах оптимизации и имеют многочисленные приложения в математической экономике, математическом моделировании транспортных потоков, теории игр и других разделах математики (см., например \cite{FaccPang_2003}). Исследования в области методики решения ВН активно продолжаются (см., например \cite{Antipin_2000}~--- \cite{Nesterov_2007}).

Наиболее известным аналогом градиентного метода для ВН является экстраградиентный метод Г.М. Корпелевич \cite{Korpelevich}. Одним из современных вариантов экстраградиентного метода является проксимальный зеркальный метод А.С. Немировского \cite{Nemirovski_2004}. Недавно Ю.Е. Нестеровым в \cite{Nesterov_2015} предложен новый алгоритм решения задач выпуклой минимизации с адаптивным выбором констант, который в случае липшицевости градиента целевой функции не требует знания этой константы Липшица (см. также раздел 5 пособия \cite{Gasn_2017}). В настоящей статье мы на базе разработанного нами критерия \eqref{4} предлагаем похожий аналог проксимального зеркального метода А.С. Немировского для решения вариационных неравенств (Алгоритм \ref{alg:SIGM}). Также мы распространяем предлагаемую методику на постановку задачи нахождения решения вариационного неравенства с неточным оракулом для оператора поля (Алгоритм \ref{alg:SIGM2}).

Для некоторого оператора $g: Q \rightarrow \mathbb{R}^n$, заданного на выпуклом компакте $Q\subset \mathbb{R}^n$ будем рассматривать {\it сильные вариационные неравенства} вида
\begin{equation}\label{1}
  \langle g(x_*),\ x_*-x\rangle \leqslant 0,
\end{equation}
где $g$ удовлетворяет условию Липшица. Отметим, что в \eqref{1} требуется найти $x_* \in Q$ (это $x_*$ и называется решением ВН) для которого
\begin{equation}\label{2}
  \max_{x\in Q}\langle g(x_*),\ x_*-x\rangle \leqslant 0.
\end{equation}
В случае монотонного поля $g$ наш подход позволяет рассматривать также {\it слабые вариационные неравенства}
\begin{equation}\label{11}
  \langle g(x),\ x_*-x\rangle \leqslant 0.
\end{equation}
Обычно в \eqref{11} требуется найти $x_* \in Q$, для которого \eqref{11} верно при всех $x \in Q$.

Начнём с некоторых вспомогательных понятий, соглашений и обозначений. Всюду далее будем полагать, что множество $Q$ выпукло и компактно в пространстве $\mathbb{R}^n$ с нормой $\|\cdot\|$ (вообще говоря, не евклидовой), а $\|\cdot\|_{\ast}$~--- сопряжённая к $\|\cdot\|$ норма. Пусть задана 1-сильно выпуклая функция $d$ относительно $\|\cdot\|$, которая дифференцируема во всех точках $x\in Q$. Можно ввести соответствующую $d$ \emph{дивергенцию Брэгмана}
\begin{equation}\label{3}
  V(x,y)=d(x)-d(y)-\langle \nabla d(y),x-y\rangle\ \forall x,y\in Q,
\end{equation}
где $\langle \cdot,\cdot\rangle$~--- скалярное произведение в $\mathbb{R}^n$.

Для решения задачи \eqref{1}~--- \eqref{2} мы предлагаем следующий адаптивный проксимальный зеркальный метод (АПЗМ). Пусть задано число $\varepsilon > 0$ (точность решения) и начальное приближение $x^0 = \mathrm{arg}\min\limits_{x\in Q} d(x) \in Q$.

Опишем ($N+1$)-ю итерацию предлагаемого алгоритма ($N = 0, 1, 2, ...$), положив изначально $N:=0$.
\begin{algorithm}[h!]
\SetAlgoNoLine
1. $N:=N+1$, $L^{N+1} := L^{N}/2$.\\
2. Вычисляем
$$y^{N+1} := \mathrm{arg}\min_{x\in Q} \left\{\left\langle g(x^N),x-x^N \right\rangle+L^{N+1}V(x,x^N)\right\},$$
$$x^{N+1} := \mathrm{arg}\min_{x\in Q}  \left\{\left\langle g(y^{N+1}),x-x^N \right\rangle+L^{N+1}V(x,x^N)\right\}.$$
3. Если верно $\left\langle g(y^{N+1})-g(x^N),y^{N+1}-x^{N+1} \right\rangle \leqslant$
\begin{equation}\label{4} \leqslant L^{N+1}V(y^{N+1},x^{N})+L^{N+1}V(x^{N+1},y^{N+1}), \end{equation}
то переходим к следующей итерации (пункт 1).\\
4. Если не выполнено \eqref{4}, то увеличиваем $L^{N+1}$ в 2 раза: $L^{N+1}:= 2L^{N+1}$ и переходим к пункту 2.\\
5. Критерий остановки метода: $\sum\limits_{k = 0}^{N-1} \frac{1}{L^{k+1}} \geqslant \frac{R^2}{\varepsilon},$ где $R^2 = \max\limits_{x \in Q} V(x, x^0).$
\caption{АПЗМ}
\label{alg:SIGM}
\end{algorithm}

\newpage

\begin{theorem}\label{th1}
После остановки Алгоритма \ref{alg:SIGM} справедлива оценка:
\begin{equation}\label{6}
\sum\limits_{k = 0}^{N-1}\frac{1}{L^{k+1}} \left\langle g(y^{k+1}),y^{k+1}-x\right \rangle \leqslant V(x,x^0)-V(x,x^N).
  \end{equation}
\end{theorem}

\begin{proof}
Непосредственно можно проверить равенства:
\begin{equation}\label{7}
   \left\langle \nabla_xV(x,x^k)\big|_{x=x^{k+1}}, \, x-x^{k+1}\right\rangle= V(x,x^k)-V(x,x^{k+1})-V(x^{k+1},x^k),
\end{equation}
\begin{equation}\label{8}
  \left\langle \nabla_xV(x,x^k)\big|_{x=y^{k+1}}, \, x-y^{k+1}\right\rangle= V(x,x^k)-V(x,y^{k+1})-V(y^{k+1},x^k).
\end{equation}
Далее, для всякого $x\in Q$ и $k=\overline{0,\ N-1}$ верны неравенства:
$$
\left\langle \nabla_x \left(\left\langle g(x^k),x-x^k \right\rangle+L^{k+1}V(x,x^k)\right)\big|_{x=y^{k+1}},\ x-y^{k+1} \right\rangle \geqslant 0,
$$
$$
\left\langle \nabla_x \left(\left\langle g(y^{k+1}),x-x^k \right\rangle+L^{k+1}V(x,x^k)\right)\big|_{x=x^{k+1}},\ x-x^{k+1} \right\rangle \geqslant 0.
$$

Поэтому
$$
\left\langle g(y^{k+1}),x^{k+1}-x \right\rangle \leqslant L^{k+1}V(x,x^k)- L^{k+1}V(x,x^{k+1})-L^{k+1}V(x^{k+1},x^k)$$
и
$$
\left\langle g(x^{k}),y^{k+1}-x \right\rangle \leqslant L^{k+1}V(x,x^{k})-L^{k+1}V(x,y^{k+1})-L^{k+1}V(y^{k+1},x^k),
$$
откуда для всякого $k=\overline{0,\ N-1}$ с учетом \eqref{4} мы имеем:
$$
\left\langle g(y^{k+1}),y^{k+1}-x \right\rangle = \left\langle g(y^{k+1}),x^{k+1}-x \right\rangle+ \left\langle g(x^{k}),y^{k+1}-x^{k+1}\right\rangle +
$$
$$
+ \left\langle g(y^{k+1})-g(x^k),y^{k+1}-x^{k+1}\right\rangle \leqslant
$$
$$
\leqslant L^{k+1}V(x,x^{k})-L^{k+1}V(x,x^{k+1})-L^{k+1}V(x^{k+1},x^{k})+L^{k+1}V(x^{k+1},x^{k})-
$$
$$
-L^{k+1}V(x^{k+1},y^{k+1})-L^{k+1}V(y^{k+1},x^{k})+L^{k+1}V(y^{k+1},x^{k})+L^{k+1}V(x^{k+1},y^{k+1}),
$$
т.е. верно
\begin{equation}\label{9}
  \frac{1}{L^{k+1}} \left\langle g(y^{k+1}),y^{k+1}-x \right\rangle \leqslant V(x,x^k)-V(x,x^{k+1}).
\end{equation}

После суммирования неравенств \eqref{9} по $k=\overline{0,\ N-1}$ получим
$$
\sum_{k=0}^{N-1} \frac{1}{L^{k+1}} \langle g(y^{k+1}),y^{k+1}-x\rangle \leqslant V(x,x^0)-V(x,x^N),
$$
что и требовалось.
\end{proof}

Допустим, что векторное поле $g$ удовлетворяет условию
\begin{equation}\label{eq10}
||g(x)-g(y)||_{\ast}\leqslant L||x-y|| \quad \forall x,y\in Q.
\end{equation}
Всюду далее для краткости будем обозначать
$$
S_N=\sum\limits_{k = 0}^{N-1}\frac{1}{L^{k+1}}.
$$
Условие \eqref{eq10} означает, что для произвольных $ x,y,z\in Q$
$$\langle g(y)-g(z), y-z \rangle \leqslant||g(y)-g(x)||_{\ast}\cdot||y-z||\leqslant L||y-x||\cdot||y-z||,$$
откуда в силу справедливого для всех $a, b \in \mathbb{R}$, $\delta>0$ неравенства
$ab\leqslant \frac{a^{2}}{2}+\frac{b^{2}}{2}$
при произвольных $x,y,z\in Q$ верно
$$\langle g(y)-g(x), y-z \rangle \leqslant\frac{L}{2}||y-x||^{2}+\frac{L}{2}||y-z||^{2}\leqslant LV(y,x)+LV(z,y).$$

Тогда ввиду \eqref{6} для всякого $x \in Q$
$$\min_{k=\overline{0,N-1}} \left\langle g(y^{k+1}), y^{k+1}-x\right\rangle \leqslant \frac{1}{S_N}\sum\limits_{k=0}^{N-1}\frac{1}{L^{k+1}} \left\langle g(y^{k+1}), y^{k+1}-x \right\rangle \leqslant$$
$$\leqslant\frac{1}{S_N}\left( V(x,x^{0})-V(x,x^{N})\right) \leqslant\frac{R^{2}}{S_N},$$
где $R^2 = \max\limits_{x\in Q}V(x,x^{0})$.

Если потребовать, чтобы
\begin{equation}\label{eq12}
\max_{x\in Q} \min_{k=\overline{0,N-1}}\left\{\left\langle g(y^{k+1}), y^{k+1}-x\right\rangle\right\} \leqslant \frac{1}{S_N} \max_{x\in Q}\sum\limits_{k=0}^{N-1}\frac{1}{L^{k+1}} \left\langle g(y^{k+1}), y^{k+1}-x \right\rangle \leqslant \varepsilon,
\end{equation}
то в случае $L^0 \leqslant 2 L$ можно оценить количество итераций работы Алгоритма \ref{alg:SIGM}. Действительно, $L^0 \leqslant 2 L$ означает, что $L^{k+1} \leqslant 2 L$ для всякого $k = \overline{{0, N-1}}$ и поэтому
$$
\frac{R^{2}}{S_N}\leqslant \varepsilon\;\text{ верно при}\;N\geqslant\frac{2LR^{2}}{\varepsilon}.
$$
Таким образом, справедлива следующая
\begin{theorem}\label{Thm_Lip}
Если выполнено условие \eqref{eq10} для поля $g$, то Алгоритм \ref{alg:SIGM} работает не более \begin{equation}\label{eqv5}
N=\left\lceil\frac{2LR^{2}}{\varepsilon}\right\rceil
\end{equation}
итераций. Если выбрать $L^{0} \leqslant 2L$, то верно \eqref{eq12}.
\end{theorem}
\begin{remark}
Если $g \not\equiv 0$, то выполнения условия $L^0 \leqslant 2 L$ можно добиться, выбрав
$$L^0 := \frac{\|g(x) - g(y)\|_{*}}{\|x - y\|} \; \text{  при  } \; g(x) \neq g(y).$$
\end{remark}
\begin{remark}
Заметим, что оценка числа итераций \eqref{eqv5} с точностью до числового множителя оптимальна для вариационных неравенств \cite{GuzNem_2015, NemYud_1979, Nemirovski_1994_95}.
\end{remark}

\begin{remark}\label{Rem_Hold}
Отметим, что похожими рассуждениями можно получить аналог теоремы \ref{Thm_Lip} для гёльдерова поля $g$. Точнее говоря, если справедливо
\begin{equation}\label{eq101}
||g(x)-g(y)||_{\ast}\leqslant L_{\nu}||x-y||^{\nu}
\end{equation}
при $\nu\in[0;1],\; x,y\in Q$, причем $L_{0}<+\infty$ (другие константы $L_{\nu}$ могут быть бесконечными), то утверждение теоремы \ref{Thm_Lip} выполняется, если
\begin{equation}
L^{0} \leqslant 2L = 2\inf_{\nu\in[0;1]} L_{\nu}\cdot\left(\frac{2L_{\nu}}{\varepsilon}\right)^{\frac{1-\nu}{1+\nu}}
\end{equation}
и
\begin{equation}\label{eqv_5.5.1} N=\left\lceil \inf_{\nu\in[0;1]} \left(\frac{2L_{\nu}R^{1+\nu}}{\varepsilon}\right)^{\frac{2}{1+\nu}}\right\rceil.
\end{equation}
\end{remark}

\begin{remark}
Для слабых ВН \eqref{3} имеем:
$$ \langle g(x),y^{k+1}-x \rangle = \langle g(y^{k+1}),y^{k+1}-x \rangle + \langle g(x)-g(y^{k+1}),y^{k+1}-x\rangle \leqslant \left \langle g(y^{k+1}),y^{k+1}-x \right\rangle,$$ поэтому критерий \eqref{eq12} можно заменить на
\begin{equation}\label{eq14}
\max_{x\in Q} \left \langle g(x), \widetilde{y}-x \right\rangle \leqslant \varepsilon, \text{ где } \widetilde{y} = \frac{\sum\limits_{k = 0}^{N-1} \frac{1}{L^{k+1}} \cdot y^{k+1}}{S_N}.
\end{equation}
Отметим, что именно оценку вида \eqref{eq14} обычно используют как критерий качества решения слабого ВН (см., например \cite{Nemirovski_2004}).
\end{remark}

В завершении отметим, что рассматриваемую в работе методику можно распространить на случай неточного оракула $\widetilde{g}(x,\delta_c,\delta_u)$ для поля $g$. Поясним смысл, который мы вкладываем в это понятие.

Будем полагать, что существует фиксированное $\delta_u > 0$ такое, что для всякого $\delta_c > 0$ существует константа $L(\delta_c) \in (0, +\infty)$, для которой $\forall x, y, z \in Q$ верно
\begin{equation}\label{eq:g_or_def}
\langle \widetilde{g}(y,\delta_c,\delta_u) - \widetilde{g}(x,\delta_c,\delta_u), y - z \rangle \leqslant L(\delta_c) \left(V(y,x) + V(z, y)\right) + \delta_c+\delta_u.
\end{equation}

\begin{remark}
Заметим, что мы здесь считаем, что $\delta_c$~--- контролируемая погрешность оракула (её возможно бесконечно уменьшать и для заданной точности решения $\varepsilon > 0$ мы всюду будем полагать, что $\delta_c = \frac{\varepsilon}{2}$). С другой стороны, погрешность $\delta_u$ мы полагаем неконтролируемой.
\end{remark}

\begin{example}
Пусть поле $g$ удовлетворяет условию Липшица и заданы его неточные значения на $Q$, т.е. существует
$\bar{\delta}_u > 0$ и в произвольной точке $x \in Q$ можно вычислить величину $\bar{g}(x)$, удовлетворяющую условию $\|\bar{g}(x)-g(x)\|_* \leqslant \bar{\delta}_u$. В таком случае $\bar{g}(x) = \widetilde{g}(x, \delta_c, \delta_u)$ удовлетворяет \eqref{eq:g_or_def}, если положить $\delta_u = \bar{\delta}_u$, а также $\delta_c = 0$ и $L(\delta_c) \equiv L$.
\end{example}

При условиях \eqref{eq:g_or_def} мы предлагаем следующую модификацию Алгоритма \ref{alg:SIGM}.
Пусть задано $\varepsilon > 0$ (точность решения) и начальное приближение\\ $x^0 = \mathrm{arg}\min\limits_{x\in Q} d(x) \in Q$.

Опишем ($N+1$)-ю итерацию предлагаемого алгоритма ($N = 0, 1, 2, ...$), положив изначально $N:= 0$.

\begin{algorithm}[h!]
\SetAlgoNoLine
1. $N:=N+1$, $L^{N+1} := L^{N}/2$.\\
2. Вычисляем:
$$y^{N+1} := \mathrm{arg}\min_{x\in Q} \left\{\left\langle \widetilde{g}\left(x^N, \frac{\varepsilon}{2}, \delta_u\right), x-x^N \right\rangle+L^{N+1}V(x, x^N)\right\},$$
$$x^{N+1} := \mathrm{arg}\min_{x\in Q}\left\{\left\langle \widetilde{g}\left(y^{N+1}, \frac{\varepsilon}{2}, \delta_u\right), x-x^N \right\rangle+L^{N+1}V(x,x^N)\right\}.$$
3. Если верно $\left\langle \widetilde{g}(y^{N+1}, \frac{\varepsilon}{2},\delta_u)- \widetilde{g}(x^N, \frac{\varepsilon}{2} ,\delta_u), y^{N+1}-x^{N+1} \right\rangle \leqslant$ $$\leqslant L^{N+1}V(y^{N+1},x^{N})+L^{N+1}V(x^{N+1},y^{N+1})+ \frac{\varepsilon}{2} + \delta_{u},$$
то переходим к следующей итерации (пункт 1).\\
4. Иначе увеличиваем в 2 раза константу $L^{N+1}:= 2L^{N+1}$ и переходим к пункту 2.\\
5. Критерий остановки метода: $\sum\limits_{k = 0}^{N-1} \frac{1}{L^{k+1}} \geqslant \frac{2 R^2}{\varepsilon}$, где $R^2 = \max\limits_{x \in Q} V(x, x^0).$
\caption{<<Неточный>> АПЗМ}
\label{alg:SIGM2}
\end{algorithm}

\newpage

Аналогично доказательству теоремы \ref{Thm_Lip} проверяется следующее утверждение.

\begin{theorem}\label{th3}
Пусть $\widetilde{g}$ удовлетворяет \eqref{eq:g_or_def}. Тогда после остановки Алгоритма \ref{alg:SIGM2} справедлива оценка:
\begin{equation}\label{6New}
\frac{1}{S_N} \sum_{k=0}^{N-1} \frac{1}{L^{k+1}} \left\langle \widetilde{g}\left(y^{k+1}, \frac{\varepsilon}{2}, \delta_u\right), y^{k+1}-x\right \rangle \leqslant  \frac{V(x,x^0)-V(x,x^N)}{S_N}+\frac{\varepsilon}{2}+\delta_u.
\end{equation}
\end{theorem}

Неравенство \eqref{6New} означает, что после работы Алгоритма \ref{alg:SIGM2} будет верно:
\begin{equation}\label{equv}
\frac{1}{S_N} \sum_{k=0}^{N-1} \frac{1}{L^{k+1}} \left\langle \widetilde{g}\left(y^{k+1}, \frac{\varepsilon}{2}, \delta_u\right),y^{k+1}-x\right \rangle \leqslant \frac{2LR^2}{N} +\frac{\varepsilon}{2}+\delta_u,
\end{equation}
где $R^2 = \max\limits_{x, y \in Q} V(x, y)$. При условии $L^{0} \leqslant 2L$ можно показать, что Алгоритм \ref{alg:SIGM2} будет работать не более
$N=\left\lceil\frac{4LR^{2}}{\varepsilon}\right\rceil$ итераций. После его остановки заведомо верно неравенство:
$$
\max_{x\in Q}\min_{k=\overline{0,N-1}}\left\langle \widetilde{g}\left(y^{k+1}, \frac{\varepsilon}{2}, \delta_u\right), y^{k+1}-x\right \rangle \leqslant
$$
$$
\leqslant \frac{1}{S_N} \max\limits_{x \in Q} \sum_{k=0}^{N-1} \frac{1}{L^{k+1}} \left\langle \widetilde{g}\left(y^{k+1}, \frac{\varepsilon}{2}, \delta_u\right),y^{k+1}-x\right \rangle \leqslant \varepsilon+\delta_u,
$$
которое может рассматриваться как критерий качества найденного решения с погрешностью $\delta_u$.

\newpage

{\bf Финансовая поддержка}. Исследование Ф.С.~Стонякина, связанное с доказательством теоремы \ref{Thm_Lip}, Алгоритмом \ref{alg:SIGM}, а также замечанием \ref{Rem_Hold}, выполнено за счёт гранта Российского научного фонда (проект 18-71-00048). Исследование А.В.~Гасникова, связанное с Алгоритмом \ref{alg:SIGM2} и доказательством теоремы \ref{th3}, выполнено при поддержке гранта Президента Росскийской Федерации для государственной поддержки молодых российских учёных-докторов наук МД-1320.2018.1, а также при поддержке программы 5 топ 100 НИУ ВШЭ.

\end{document}